\documentclass[reqno,b5paper]{amsart}
\usepackage{amsmath,amsfonts,amssymb,amsthm,latexsym}
\usepackage[mathscr]{eucal}

\makeatletter
\@namedef{subjclassname@2020}{%
  \textup{2020} Mathematics Subject Classification}
\makeatother

\newtheorem{theorem}{Theorem}[section]
\newtheorem{corollary}[theorem]{Corollary}

\theoremstyle{definition}

\newtheorem{remark}[theorem]{Remark}
\newtheorem{example}[theorem]{Example}
\newtheorem{problem}{Problem}

\numberwithin{equation}{section}

\DeclareMathOperator{\lcm}{lcm}
\newcommand{\Mod}[1]{\ (\mathrm{mod}\ #1)}

\begin{document}

\title[Sums of divisors on arithmetic progressions]{Sums of divisors on arithmetic progressions}

\author[P. Pongsriiam]{Prapanpong Pongsriiam}
\newcommand{\acr}{\newline\indent}
\address{Department of Mathematics\acr
Faculty of Science, Silpakorn University\acr
Nakhon Pathom \acr
73000\acr
THAILAND}
\email{prapanpong@gmail.com, pongsriiam\_p@silpakorn.edu}


\begin{abstract}
For each $s\in \mathbb R$ and $n\in \mathbb N$, let $\sigma_s(n) = \sum_{d\mid n}d^s$. In this article, we give a comparison between $\sigma_s(an+b)$ and $\sigma_s(cn+d)$ where $a$, $b$, $c$, $d$, $s$ are fixed, the vectors $(a,b)$ and $(c,d)$ are linearly independent over $\mathbb Q$, and $n$ runs over all positive integers. For example, if $|s|\leq 1$, $a, b, c, d\in \mathbb N$ are fixed and satisfy certain natural conditions, then 
$$
\sigma_s(an+b) < \sigma_s(cn+d)\quad\text{ for all $n\leq M$}
$$
where $M$ may be arbitrarily large, but in fact $\sigma_s(an+b) - \sigma_s(cn+d)$ has infinitely many sign changes. The results are entirely different when $|s|>1$, where the following three cases may occur: 
\begin{itemize}
\item[(i)] $\sigma_s(an+b) < \sigma_s(cn+d)$ for all $n\in \mathbb N$; 
\item[(ii)] $\sigma_s(an+b) < \sigma_s(cn+d)$ for all $n\leq M$ and $\sigma_s(an+b) > \sigma_s(cn+d)$ for all $n\geq M+1$; 
\item[(iii)] $\sigma_s(an+b) - \sigma_s(cn+d)$ has infinitely many sign changes.
\end{itemize}
We also give several examples and propose some problems.
\end{abstract}

\maketitle

\section{Introduction}

Perhaps, the most interesting inequalities in number theory are those between primes and functions whose values depend on primes. For example, functions $f$ and $g$ may satisfy $f(x)<g(x)$ on a large interval $[1,M]$, but in fact $f(x)-g(x)$ has infinitely many sign changes as $x\to\infty$; and the smallest number $x_0$ for which $f(x_0)>g(x_0)$ exists in theory but is very difficult to find explicitly. Although it might be better to focus only on those $f$ and $g$ which are asymptotic and satisfy some kinds of Chebyshev's bias phenomenon, we do not require these conditions here. In fact, we are interested in the inequalities between $f(x)$ and $g(x)$ and the number of sign changes in $f(x)-g(x)$ as $x\to\infty$. 

In this article, we will be concerned with the comparison between the values of the function $\sigma_s$ over arithmetic progressions. Throughout, $s$ is a real number, $m$ and $n$ are positive integers, $p$ is a prime, $p_n$ is the $n$th prime, $\varphi(n)$ is the number of $m\leq n$ with $\gcd(m,n)=1$, $\sigma_s(n) = \sum_{d\mid n}d^s$, $\tau(n) = \sigma_0(n)$, $\sigma(n) = \sigma_1(n)$, and $\zeta$ is the Riemann zeta function which is given by $\zeta(s) = \sum_{n=1}^\infty n^{-s}$ for $s>1$. We sometimes simply write $(m,n)$ to denote $\gcd(m,n)$.

Jarden \cite[p.\ 65]{Jar} observed that $\varphi(30n+1) > \varphi(30n)$ for all $n\leq 10^5$ and further computations will show us that this inequality continues to hold for all $n$ up to a billion. However, contrary to the numerical evidence, Newman \cite{New} proved that there are infinitely many $n$ such that $\varphi(30n+1) < \varphi(30n)$; the smallest such $n$ for which this inequality holds has $1116$ digits and was given explicitly by Martin \cite{Mart}, namely, $n = (z-1)/30$ where $z = \left(\prod_{j=4}^{383}p_j\right)p_{385}p_{388}$. Aldaz, Bravo, Guti\'errez, and Ubis \cite{ABGU} extended Newman's theorem to the following form: if $a$, $b$, $c$, $d$ are nonnegative integers, $a, c>0$, and $ad-bc\neq 0$, then 
$$
\liminf_{n\to\infty}\frac{\varphi(an+b)}{\varphi(cn+d)} = 0\quad\text{and}\quad \limsup_{n\to\infty}\frac{\varphi(an+b)}{\varphi(cn+d)} = \infty.
$$
The assumption $ad-bc\neq 0$ is necessary because they \cite{ABGU} also showed that if $ad-bc=0$, then $R\leq \varphi(an+b)/\varphi(cn+d)\leq M$ for some positive real numbers $R$ and $M$, and so the limits infimum and supremum are not $0$ and $\infty$, respectively. This gives a complete picture for the comparison between $\varphi(an+b)$ and $\varphi(cn+d)$.

However, as far as we are aware, this kind of investigation has not been done for $\sigma_s$. We find a similar (but reverse) inequality, namely, 
$$
\sigma(30n+1) < \sigma(30n)\quad\text{for all $n\leq 10^7$,}
$$
yet there are infinitely many $n$ for which $\sigma(30n+1) > \sigma(30n)$. In general, if $-1\leq s\leq 1$, then the results for $\sigma_s$ are similar to those of $\varphi$ (see Theorem \ref{sumofdfthm203}, Corollary \ref{sumofdfthm204}, and Examples \ref{example3.7race} and \ref{example3.9race}). Nevertheless, if $|s| > 1$, then the answers are completely different. First of all, we show in Theorem \ref{sumofdfthm2} that $\sigma_s(an+b)/\sigma_s(cn+d)$ is bounded away from $0$ and $\infty$. Secondly, comparing the size of $\sigma_s(an+b)$ and $\sigma_s(cn+d)$, the following three inequality examples may occur: 
\begin{itemize}
\item[(IE1)] Always win or always lose: $\sigma_s(an+b) < \sigma_s(cn+d)$ for all $n\in \mathbb N$ (see Theorems \ref{newthmB}, \ref{newthmB2.10.5} and \ref{sumofdfthm3.8}, and Examples \ref{sumofdfthm200exam3.2} and \ref{example3.9race}).
\item[(IE2)] Change signs exactly once: $\sigma_s(an+b) < \sigma_s(cn+d)$ for an arbitrarily long string of consecutive integers $n = 1, 2, \ldots, M$, and $\sigma_s(an+b) > \sigma_s(cn+d)$ for all $n\geq M+1$ (see Theorem \ref{sumofdfthmA} and Example \ref{sumofdfexamA4}).
\item[(IE3)] Change signs infinitely often: 
$$
\sigma_s(an+b) < \sigma_s(cn+d)\;\text{ and}\;\sigma_s(am+b) > \sigma_s(cm+d)
$$
 for infinitely many $m$ and $n$ (see Theorems \ref{newthmA} and \ref{newthmC}).
\end{itemize}  
The inequalities in (IE1), (IE2), and (IE3) may be considered as the $0$-type, $1$-type, and $\infty$-type inequalities according to the number of sign changes in $\sigma_s(an+b) - \sigma_s(cn+d)$ as $a$, $b$, $c$, $d$ are fixed and $n$ runs over all positive integers. We do not know whether or not there exists an example for a $k$-type inequality for each positive integer $k>1$ (see also Problem \ref{problem12sign}).

By changing the role of $a$, $b$, $c$, $d$, we can replace the signs $<$ by $>$ and $>$ by $<$ in (IE1), (IE2), and (IE3). In addition, to  avoid triviality or unnecessary complications, we focus our study on the cases where $a$, $b$, $c$, $d$ are nonnegative integers, $a, c>0$, and $ad-bc\neq 0$. Note that the assumption $ad-bc \neq 0$ is equivalent to $(a,b)$ and $(c,d)$ are linearly independent over $\mathbb Q$; for if $ad-bc = 0$, then $\frac1a(a,b) - \frac1c(c,d) = (0,0)$ and if $(a,b) = (q_1/q_2)(c,d)$ for some $q_1\in \mathbb Z$, $q_2\in \mathbb N$, then $ad-bc = \left(q_1c/q_2\right)d - \left(q_1d/q_2\right)c = 0$.

In the next section, we prove the main results and show several examples. We also propose some possible research problems and give some related references. We do not plan to solve these problems soon; we will do it in the future but we do not mind if the reader will solve them. 

\section{Main Results}\label{Main results}
 
 Throughout, let $a$, $b$, $c$, $d$ be nonnegative integers with $a, c>0$. We first deal with the case $ad=bc$ in Theorem \ref{sumofdfthm200} and Example \ref{sumofdfthm200exam3.2}. After that, we restrict ourselves to the case $ad\neq bc$ and give general criteria and several examples for (IE1), (IE2), and (IE3) introduced earlier. Recall that the function $\sigma_s$ is multiplicative for every $s\in \mathbb R$, and observe that if $ad = bc$, then there are positive integers $r_1$ and $r_2$ such that $r_1(a,b) = r_2(c,d)$. We have the following result.
\subsection{The case $ad = bc$ and $s\in \mathbb R$}
\begin{theorem}\label{sumofdfthm200}
Let $s\in \mathbb R$ and let $a$, $b$, $c$, $d$ be nonnegative integers, $a, c>0$, and $ad = bc$. Let $r_1$ and $r_2$ be positive integers such that $r_1(a,b) = r_2(c,d)$. Then for every $n\geq 1$, 
$$
\frac{r_2^s}{\sigma_s(r_1)} \leq \frac{\sigma_s(an+b)}{\sigma_s(cn+d)} \leq \frac{\sigma_s(r_2)}{r_1^s}.
$$
\end{theorem}
\begin{proof}
Observe that if $p$ is a prime and $\alpha$, $\beta$ are nonnegative integers, then $\sigma_s(p^\alpha)\sigma_s(p^\beta)$ is equal to $(1+p^s+\cdots+p^{\alpha s})(1+p^s+\cdots+p^{\beta s})$, which is larger than or equal to 
$$
(1+p^s+\cdots+p^{\alpha s})+p^{\alpha s}(p^s+p^{2s}+\cdots+p^{\beta s}) = 1+p^s+p^{2s}+\cdots+p^{(\alpha+\beta)s} = \sigma_s(p^{\alpha+\beta}).
$$
Similarly, $\sigma_s(p^{\alpha+\beta}) \geq p^{\alpha s}\sigma_s(p^\beta)$. From this and the multiplicativity of $\sigma_s$, it follows that $\sigma_s(m)\sigma_s(n) \geq \sigma_s(mn) \geq m^s\sigma_s(n)$ for all $m, n\in \mathbb N$. Therefore for every $n\geq 1$,
\begin{align*}
\sigma_s(r_1)\sigma_s(an+b) &\geq \sigma_s(r_1(an+b)) \geq r_1^s\sigma_s(an+b),\\
\sigma_s(r_2)\sigma_s(cn+d) &\geq \sigma_s(r_2(cn+d)) \geq r_2^s\sigma_s(cn+d).
\end{align*}
Since the middle terms of each inequality are the same, these lead to the desired result.
\end{proof}

\begin{example}\label{sumofdfthm200exam3.2}
If $a, b, c, d$ are integers as given in Theorem \ref{sumofdfthm200}, then it is easy to construct an example such that $\sigma_s(an+b) < \sigma_s(cn+d)$ for all $n\in \mathbb N$. For example, $\sigma_s(2n) < \sigma_s(4n)$ for all $n\in \mathbb N$ and $s\in \mathbb R$. In general, if $0<a<c$ and $a\mid c$, then every divisor of $a$ is also a divisor of $c$, while $c\mid c$ but $c\nmid a$, which implies $\sigma_s(an) < \sigma_s(cn)$ for all $n\in \mathbb N$ and $s\in \mathbb R$. 
\end{example}

\begin{problem}
By Example \ref{sumofdfthm200exam3.2}, we know that (IE1) exists for infinitely many $a, b, c, d$ satisfying $ad = bc$. It is natural to ask whether or not (IE2) and (IE3) can occur in this case. For example, if $M$ is a given positive integer, can we always find $s\in \mathbb R$ and $a, b, c, d$ satisfying the assumption of Theorem \ref{sumofdfthm200} such that $\sigma_s(an+b) < \sigma_s(cn+d)$ for $n\leq M$ and $\sigma_s(an+b) > \sigma_s(cn+d)$ for all $n\geq N$ where $N\geq M+1$? Can we choose $N$ to be $M+1$ for each $M$? If $s$ is a given real number, are there nonnegative integers $a, b, c, d$ satisfying $a, c>0$ and $ad = bc$ such that $\sigma_s(an+b) - \sigma_s(cn+d)$ changes sign infinitely many times? 
\end{problem}

\begin{problem}
Does the converse of Example \ref{sumofdfthm200exam3.2} hold? In other words, if $s\in \mathbb R$, $a, b, c, d$ are nonnegative integers, $a, c>0$, $ad=bc$, and 
$$
\sigma_s(an+b)<\sigma_s(cn+d)\quad\text{for all $n\in \mathbb N$,}
$$
can we conclude that $a\mid c$? or perhaps $(a,b)\mid(c,d)$?
\end{problem}

\begin{remark}\label{remark2.3}
Before proceeding to the next subsection, let us recall that if $1+a_n>0$ for all $n\in \mathbb N$, then the infinite product $\prod_{n=1}^\infty(1+a_n)$ is said to be convergent if $\lim_{k\to\infty}\prod_{n=1}^k(1+a_n)$ exists and is not zero. If the limit converges to zero (or diverges to $\infty$), then the infinite product is said to diverge to zero (or diverge to $\infty$), respectively. In addition, it is well-known that if $\sum_{n=1}^\infty |a_n|$ converges, then $\prod_{n=1}^\infty(1+|a_n|)$ converges; if $\sum_{n=1}^\infty |a_n| = \infty$, then $\prod_{n=1}^\infty(1+|a_n|)$ also diverges to $\infty$. 
\end{remark}

We will often refer to the following condition:
\begin{equation}\label{conditionA}
\text{(A) $a, b, c, d$ are nonnegative integers, $a, c>0$, and $ad\neq bc$.}
\end{equation}

\subsection{The case $ad\neq bc$ and $|s|\leq 1$}

\begin{theorem}\label{sumofdfthm201}
Let $a, b, c, d$ satisfy the condition {\rm(A)} given in \eqref{conditionA}. Then there exists a strictly increasing sequence $(n_k)$ of positive integers that does not depend on $s$ and  
$$
{\rm(i)} \lim_{k\to\infty}\frac{\sigma_s(an_k+b)}{\sigma_s(cn_k+d)} = 0\quad\text{for all $s\in [0,1]$}.
$$
In particular, if $0\leq s\leq 1$, then   
$$
\text{{\rm(ii)} $\displaystyle \liminf_{n\to\infty}\frac{\sigma_s(an+b)}{\sigma_s(cn+d)} = 0$ \quad and\quad  {\rm(iii)} $\displaystyle \limsup_{n\to\infty}\frac{\sigma_s(an+b)}{\sigma_s(cn+d)} = \infty$.}
$$
\end{theorem}
\begin{proof}
We modify the idea of Newman \cite{New} and Aldaz et. al \cite{ABGU}. Recall that $p$ is always a prime and $p_n$ is the $n$th prime. Let $D = 2\tau(|ad-bc|)$ and for each $k\geq 1$, let $m_k$ be the product of all primes $p\leq p_k$ with $p\nmid c$. Since $ad-bc\neq 0$, $D$ is well-defined. In addition, $m_k$ is a finite product, $m_k\leq p_k^k$, and $m_k\to\infty$ as $k\to\infty$. The following observation will be used throughout this article sometimes without reference. We have, for any $s\geq 0$,
\begin{equation}\label{sumofdfeq1}
\frac{\sigma_s(n)}{n^s} = \sum_{d\mid n}\frac{1}{d^s} = \prod_{p^\alpha\parallel n}\left(1+\frac{1}{p^s}+\frac{1}{p^{2s}}+\cdots +\frac{1}{p^{\alpha s}}\right)\geq \prod_{p\mid n}\left(1+\frac{1}{p^s}\right).
\end{equation}
Next, we will construct a strictly increasing sequence $(n_k)$ of positive integers such that 
\begin{equation}\label{sumofdfeq3}
\frac{\sigma_s(an_k+b)}{(an_k+b)^s} \leq D\quad\text{and}\quad \lim_{k\to\infty}\frac{(cn_k+d)^s}{\sigma_s(cn_k+d)} = 0\quad\text{for all $s\in [0,1]$}.
\end{equation}

Since $\gcd(m_k,c) = 1$, there exists a pair of integers $n_0$, $y_0$ which is a solution to the Diophantine equation $cn-m_ky = -d$, and all solutions are given by $n = n_0+m_kt$ and $y = y_0+ct$ where $t\in \mathbb Z$ is arbitrary. We keep in mind that the integers $n_0$, $y_0$ may depend on $k$ and so $n$ and $y$ depend on $k$ and $t$. Then $an+b = a(n_0+m_kt)+b = (an_0+b)+am_kt = \delta(A+Bt)$ where $\delta = \gcd(an_0+b,am_k)$, $A = (an_0+b)/\delta$, and $B = am_k/\delta$. Then $\delta$ divides $c(an_0+b) - y_0(am_k) = a(cn_0-m_ky_0)+bc = bc-ad$. Since $(A,B) = 1$, we obtain by Dirichlet's theorem for primes in arithmetic progressions that there are infinitely many $t\in \mathbb Z^+$ such that $A+Bt$ is a prime. So we choose $t\in \mathbb Z^+$ so that $A+Bt$ is a prime larger than $|bc-ad|\geq \delta$. Therefore for each $k\geq 1$, we can choose a large positive integer $t_k$ and a pair of positive integers $n_k = n_{k, t_k}$, $y_k = y_{k, t_k}$ satisfying $cn_k-m_ky_k = -d$, $an_k+b = \delta q_k$ where $q_k$ is a prime larger than $\delta$, $n_{k+1}-n_k > 0$, and $q_{k+1}-q_k > 0$ for all $k$. Now that we have the sequence $(n_k)$, it remains to prove \eqref{sumofdfeq3}. So let $s\in [0,1]$. Since $s\leq 1$, the series $\sum_pp^{-s}$ diverges to $\infty$. In addition, by \eqref{sumofdfeq1} and the fact that $cn_k+d = m_ky_k$, we obtain 
$$
\frac{(cn_k+d)^s}{\sigma_s(cn_k+d)} \leq \prod_{p\mid m_ky_k}\left(1+\frac{1}{p^s}\right)^{-1} \leq \prod_{p\mid m_k}\left(1+\frac{1}{p^s}\right)^{-1} = \prod_{\substack{p\leq p_k\\p\nmid c}}\left(1+\frac{1}{p^s}\right)^{-1},
$$
which diverges to zero as $k\to\infty$. Hence the second part of \eqref{sumofdfeq3} is verified. Since $an_k+b = \delta q_k$, $\delta\mid ad-bc$, $q_k>\delta$, and $\sigma_s$ is multiplicative, we have  
$$
\frac{\sigma_s(an_k+b)}{(an_k+b)^s} = \frac{\sigma_s(\delta)}{\delta^s}\frac{\sigma_s(q_k)}{q_k^s} = \left(\sum_{u\mid \delta}\frac{1}{u^s}\right)\left(1+\frac{1}{q_k^s}\right)\leq 2\tau(\delta)\leq D,
$$
which is the first part of \eqref{sumofdfeq3}. By writing,
\begin{equation}\label{sumofdfeqABC}
\frac{\sigma_s(an+b)}{\sigma_s(cn+d)} = \frac{\sigma_s(an+b)}{(an+b)^s}\cdot \left(\frac{an+b}{cn+d}\right)^s\cdot \frac{(cn+d)^s}{\sigma_s(cn+d)},
\end{equation}
substituting $n=n_k$, and applying \eqref{sumofdfeq3}, we obtain (i). Then (ii) follows immediately from (i). Since (ii) holds for all $a, b, c, d$, we can interchange the role of $a, c$ and $b,d$ to obtain (iii). This completes the proof.
\end{proof}

We can extend Theorem \ref{sumofdfthm201} to the case $s\in [-1,1]$ as follows.
\begin{theorem}\label{sumofdfthm203}
Let $a$, $b$, $c$, $d$ satisfy the condition {\rm(A)} and $(n_k)$ the sequence constructed in the proof of Theorem \ref{sumofdfthm201}. Then $\sigma_s(an_k+b)/\sigma_s(cn_k+d)$ converges to zero as $k\to \infty$ for all $s\in [-1,1]$. In addition, there exists a sequence $(m_k)$ such that $\sigma_s(am_k+b)/\sigma_s(cm_k+d)\to \infty$ as $k\to \infty$ for all $s\in [-1,1]$. In particular, if $|s|\leq 1$, then 
$$
\liminf_{n\to\infty}\frac{\sigma_s(an+b)}{\sigma_s(cn+d)} = 0\quad\text{and}\quad \limsup_{n\to\infty}\frac{\sigma_s(an+b)}{\sigma_s(cn+d)} = \infty.
$$
\end{theorem}
\begin{proof}
If $0\leq s\leq 1$, then this follows from Theorem \ref{sumofdfthm201}. So assume that $-1\leq s <0$. Let $r=-s$. Then $0<r\leq 1$ and for each $m\in \mathbb N$, we have 
$$
\frac{\sigma_r(m)}{m^r} = \sum_{d\mid m}\frac{1}{d^r} = \sum_{d\mid m}d^s = \sigma_s(m).
$$
Therefore 
$$
\frac{\sigma_s(an_k+b)}{\sigma_s(cn_k+d)} = \frac{\sigma_r(an_k+b)}{\sigma_r(cn_k+d)}\left(\frac{cn_k+d}{an_k+b}\right)^r.
$$
Applying Theorem \ref{sumofdfthm201} to the right-hand side of the above equation, we obtain that the left-hand side converges to zero as $k\to\infty$, as required. Since this is true for all $a, b, c, d$, we can interchange the role of $a, b, c, d$ to obtain the sequence $(m_k)$ with the desired property. The rest follows immediately. 
\end{proof}
\begin{corollary}\label{sumofdfthm204}
If $-1\leq s_1<s_2<\cdots<s_\ell\leq 1$ and $a$, $b$, $c$, $d$ satisfy the condition {\rm(A)}, then there are infinitely many $m, n\in \mathbb N$ such that 
$$
\sigma_s(am+b) < \sigma_s(cm+d)\;\text{and}\;\sigma_s(an+b) > \sigma_s(cn+d)
$$
for all $s\in \{s_1, s_2, \ldots, s_\ell\}$.
\end{corollary}
\begin{proof}
Since the sequences $(n_k)$ and $(m_k)$ in Theorem \ref{sumofdfthm203} can be used for all $s\in [-1,1]$, this corollary follows immediately from Theorem \ref{sumofdfthm203}.
\end{proof}
\begin{example}\label{example3.7race}
By running the computation in a computer, we find that 
$$
\sigma_s(30n+1) < \sigma_s(30n)\quad\text{for all $s\in \{-1, 0, 1/2, 1\}$ and $n\leq 10^6$.}
$$
However, by Corollary \ref{sumofdfthm204}, there are infinitely many $m\in \mathbb N$ such that 
$$
\sigma_s(30m+1) > \sigma_s(30m)\quad\text{for all $s\in \{-1, 0, 1/2, 1\}$.}
$$
The smallest $m$ for which $\sigma_{1/2}(30m+1) > \sigma_{1/2}(30m)$ is $m = 2338703$. When $s\in \{-1, 0, 1\}$ the smallest such $m$ seems to be very large. 
\end{example}
\begin{problem}
It may be interesting to find the smallest positive integer $n$ for which the following inequalities hold: 
$$
\sigma(6n+1) > \sigma(6n), \sigma(30n+1) > \sigma(30n), \sigma(210n+1) > \sigma(210n),\;\text{etc.}
$$
Adjusting Martin's method \cite{Mart} may lead to such the integer $n$. In general, suppose $a, b, c, d$ are fixed and satisfy the condition {\rm(A)}. Let $f:[-1,1]\to \mathbb N$ be defined by $f(s) = f_{\sigma,a,b,c,d}(s) = n_s$ be the smallest positive integer for which $\sigma_s(an_s+b)-\sigma_s(cn_s+d)$ changes sign. How is $f$ behave? Is $f$ increasing on $[1/2,1]$? Definitely, the answer depends on $a, b, c, d$. Are there $a, b, c, d$ such that $f$ is increasing on $[-1,1]$ or on $[0,1]$? Is $f$ a step function? Many questions can be asked. We leave them to the reader's curiosity.  For $2\leq k\leq 10$, let $g(k) = f_{\sigma,6,1,6,0}\left(\frac{k-1}{k}\right)$ be the smallest $n$ such that $\sigma_{\frac{k-1}{k}}(6n+1) > \sigma_{\frac{k-1}{k}}(6n)$. Then, by running the computation in a computer, we find that $g(2) = 379$, $g(3) = 5839$, $g(4) = 95929 = g(5)$, $g(6) = 326159 = g(7)$, $g(8) = 2198029 = g(9)$, and $g(10) = 7813639$. This is a numerical data which suggests that $g$ may be increasing on $[1/2,1]$. 
\end{problem}
 \begin{remark}
When $|s| > 1$, the results are entirely different. First of all, unlike the results for $|s|\leq 1$ in Theorem \ref{sumofdfthm203}, it does not matter if $ad-bc$ is zero or nonzero, $\sigma_s(an+b)/\sigma_s(cn+d)$ is always bounded away from zero and infinity when $|s|>1$ as shown in the next theorem.
\end{remark}

\subsection{The first case for $|s|>1$}
 
\begin{theorem}\label{sumofdfthm2}
Let $|s|>1$ and let $a$, $b$, $c$, $d$ be nonnegative integers with $a, c>0$. Then there are positive real numbers $R$ and $M$ such that 
$$
\frac{R^{|s|}}{\zeta(|s|)} \leq \frac{\sigma_s(an+b)}{\sigma_s(cn+d)} \leq \zeta(|s|)M^{|s|}.
$$
\end{theorem}
\begin{proof}
Since $(an+b)/(cn+d)$ converges to $a/c>0$ as $n\to \infty$, there are $R_1, M_1>0$ such that $R_1\leq (an+b)/(cn+d)\leq M_1$ for all $n\geq 1$. Let $R = \min\{1, R_1\}$ and $M = \max\{1,M_1\}$. If $s>1$ and $n\in \mathbb N$, then 
$$
R^s \leq \frac{(an+b)^s}{(cn+d)^s} \leq M^s
$$
 and 
\begin{equation}\label{sumofdfthm2eq1}
1\leq \frac{\sigma_s(n)}{n^s} = \sum_{d\mid n}\frac{1}{d^s} \leq \sum_{d=1}^{\infty}\frac{1}{d^s} = \zeta(s).
\end{equation}
So if $s>1$ and $n\in \mathbb N$, we write $\sigma_s(an+b)/\sigma_s(cn+d)$ as in \eqref{sumofdfeqABC} and apply \eqref{sumofdfthm2eq1} to  obtain the desired result. If $s<-1$ and $n\in \mathbb N$, then 
$$
1\leq \sigma_s(n) = \sum_{d\mid n}\frac{1}{d^{-s}} \leq \zeta(-s)
$$
 and therefore 
$$
\frac{R^{-s}}{\zeta(-s)} \leq \frac{1}{\zeta(-s)} \leq \frac{\sigma_s(an+b)}{\sigma_s(cn+d)} \leq \zeta(-s) \leq \zeta(-s)M^{-s}.
$$ 
This completes the proof. 
\end{proof}

Various situations where (IE1) may occur are given in Theorems \ref{newthmB} and \ref{sumofdfthm3.8}. The integer $N$ in Theorem \ref{newthmB} may or may not be $1$ but it can be chosen to be $1$ as shown in Theorems \ref{newthmB2.10.5} and \ref{sumofdfthm3.8}. From this point on, we restrict ourselves to the case $s>1$ and leave the study of $s<-1$ to the interested reader. 

\begin{theorem}\label{newthmB}
Suppose $a$ and $c$ are distinct positive integers.
\begin{itemize}
\item[{\rm(i)}] If $a>c$ and $b, d$ are nonnegative integers, then there are $N\in \mathbb N$ and a real number $s_0>1$ such that $\sigma_s(an+b) > \sigma_s(cn+d)$ for all $s\geq s_0$ and $n\geq N$. 
\item[{\rm(ii)}] If $a<c$ and $b, d$ are nonnegative integers, then there are $N\in \mathbb N$ and a real number $s_0>1$ such that $\sigma_s(an+b) < \sigma_s(cn+d)$ for all $s\geq s_0$ and $n\geq N$.
\end{itemize}
\end{theorem}
\begin{proof}
As usual, (ii) follows from (i) by changing the role of $a$, $c$ and $b$, $d$. So we only need to prove (i). Let $\varepsilon = (a/c-1)/2$. Then $1<1+\varepsilon<a/c$. Since $(an+b)/(cn+d)\to a/c$ as $n\to \infty$, there exists $N\in \mathbb N$ such that $an+b>(1+\varepsilon)(cn+d)$ for all $n\geq N$. In addition, $\zeta(s)\to 1$ as $s\to \infty$, so there exists $s_0>1$ such that $\zeta(s) < 1+\varepsilon$ for all $s\geq s_0$. For $s\geq s_0$ and $n\geq N$, the first quotient on the right-hand side of \eqref{sumofdfeqABC} is larger than $1$, and therefore 
\begin{align*}
\frac{\sigma_s(an+b)}{\sigma_s(cn+d)} &> \left(\frac{an+b}{cn+d}\right)^s\frac{(cn+d)^s}{\sigma_s(cn+d)} > \frac{(1+\varepsilon)^s}{\sum_{k\mid cn+d}k^{-s}}\\
&> \frac{(1+\varepsilon)^s}{\zeta(s)}> (1+\varepsilon)^{s-1}> 1.
\end{align*}
This completes the proof.
\end{proof}

Adjusting Theorem \ref{newthmB} a little, we can take $N=1$ as follows.

\begin{theorem}\label{newthmB2.10.5}
Let $a, b, c, d$ be integers. Then the following statements hold. 
\begin{itemize}
\item[{\rm(i)}] If $a>c>0$ and $b\geq d\geq 0$, then there exists $s_0>1$ such that $\sigma_s(an+b) > \sigma_s(cn+d)$ for all $s\geq s_0$ and $n\geq 1$.
\item[{\rm(ii)}] If $0<a<c$ and $0\leq b\leq d$, then there exists $s_0>1$ such that $\sigma_s(an+b) < \sigma_s(cn+d)$ for all $s\geq s_0$ and $n\geq 1$.
\end{itemize}
\end{theorem}
\begin{proof}
We only need to prove (i). Following the proof of Theorem \ref{newthmB}, let $\varepsilon$ be a positive real number satisfying 
$$
\varepsilon < \min\left\{\frac1c,\frac{b-d+1}{c+d}\right\}.
$$
Since $a\geq c+1$ and $\varepsilon c<1$, we have $a-c-\varepsilon c\geq 1-\varepsilon c > 0$. Therefore, for every $n\geq 1$, 
\begin{align*}
(an+b)-(1+\varepsilon)(cn+d) &= (a-c-\varepsilon c)n+b-d-\varepsilon d \\
&\geq (1-\varepsilon c)+b-d-\varepsilon d\\
&= (1+b-d)-\varepsilon(c+d) > 0.
\end{align*}
Thus, $an+b > (1+\varepsilon)(cn+d)$ for all $n\geq 1$, that is, we can take $N$ in the proof of Theorem \ref{newthmB} to be $N = 1$. The rest is the same and the proof is complete. 
\end{proof}

The difference between the assumptions of Theorem \ref{newthmB} and \ref{newthmB2.10.5}(i) is that $b\geq d\geq 0$ in the latter while $b, d$ are any given nonnegative integers in the former. In both theorems, the real number $s_0$ is selected and depends on the given integers $a, b, c, d$. In the next theorem, $s_0$ may be chosen to be independent of $a, b, c, d$.

\begin{theorem}\label{sumofdfthm3.8}
Let $s_0>1$, $a$, $b$, $c$, $d$ satisfy the condition {\rm(A)} and $ad>bc$. If either $a^{s_0}\zeta(s_0) < c^{s_0}$ or $1\leq a<c(1-1/s_0)$, then $\sigma_s(an+b) < \sigma_s(cn+d)$ for all $n\in \mathbb N$ and for all $s\geq s_0$. In particular, for each $s_0>1$, there are infinitely many nonnegative integers $a$, $b$, $c$, $d$ such that $0< a < c$, $ad-bc > 0$, and $\sigma_s(an+b) < \sigma_s(cn+d)$ for all $s\geq s_0$ and $n\in \mathbb N$. 
\end{theorem}
\begin{proof}
We first recall that for $s>1$ and $x>0$, we have 
$$
\sum_{n\leq x}\frac{1}{n^s} \leq 1+\int_1^x\frac{1}{t^s}dt = \frac{s}{s-1}+\frac{x^{1-s}}{1-s} \leq \frac{s}{s-1},
$$
which implies $\zeta(s) \leq s/(s-1)$. Now suppose that either $a^{s_0}\zeta(s_0) < c^{s_0}$ or $1\leq a<c(1-1/s_0)$ holds. Since $ad-bc>0$, we see that 
$$
\frac{an+b}{cn+d} < \frac ac < 1\quad\text{ for every $n\in \mathbb N$.}
$$
Therefore if $a^{s_0}\zeta(s_0) < c^{s_0}$, $n\in \mathbb N$, and $s\geq s_0$, then 
\begin{align*}
\frac{\sigma_s(an+b)}{\sigma_s(cn+d)} &= \frac{\sigma_s(an+b)}{(an+b)^s}\left(\frac{an+b}{cn+d}\right)^s\frac{(cn+d)^s}{\sigma_s(cn+d)}\\
&\leq \zeta(s)\left(\frac ac\right)^s \leq \zeta(s_0)\left(\frac ac\right)^{s_0} < 1.
\end{align*}
 If $1\leq a<c(1-1/s_0)$, then the above implies that for $s\geq s_0$ and $n\in \mathbb N$,
$$
\frac{\sigma_s(an+b)}{\sigma_s(cn+d)} \leq \zeta(s_0)\left(\frac ac\right)^{s_0} \leq \frac{s_0}{s_0-1}\left(\frac ac\right) < 1.
$$
In any case, we have $\sigma_s(an+b) < \sigma_s(cn+d)$ for all $n\in \mathbb N$ and $s\geq s_0$. It remains to show that there are infinitely many nonnegative integers $a$, $b$, $c$, $d$ satisfying all the required condition. Since $s_0>1$ is given, we can find a large positive integer $c$ such that $c(1-1/s_0) > 1$. Then there exists a positive integer $a<c(1-1/s_0)$. Now we can choose any $b\geq 0$ and then select any $d$ satisfying $d>bc/a$. Then $0 < a < c$, $ad-bc > 0$, $1\leq a<c(1-1/s_0)$, and $\sigma_s(an+b) < \sigma_s(cn+d)$ for all $s\geq s_0$ and $n\in \mathbb N$. This completes the proof.
\end{proof}

We can change the condition $ad>bc$ in Theorem \ref{sumofdfthm3.8} and modify the proof to obtain the following result.
\begin{theorem}\label{sumofdfthm3.82.11.5}
Let $s_0>1$, $a$, $b$, $c$, $d$ satisfy the condition {\rm(A)}, $ad < bc$, and $a+b < c+d$. If either $(a+b)^{s_0}\zeta(s_0) < (c+d)^{s_0}$ or $a+b < (c+d)(1-1/s_0)$, then $\sigma_s(an+b) < \sigma_s(cn+d)$ for all $n\in \mathbb N$ and for all $s\geq s_0$. In particular, for each $s_0>1$, there are infinitely many nonnegative integers $a$, $b$, $c$, $d$ having the above properties. 
\end{theorem}
\begin{proof}
Since $ad<bc$ and $a+b < c+d$, we have 
$$
\frac ac < \frac{an+b}{cn+d} \leq \frac{a+b}{c+d} < 1\quad \text{for all $n\geq 1$}.
$$
Following the proof of Theorem \ref{sumofdfthm3.8}, if $(a+b)^{s_0}\zeta(s_0) < (c+d)^{s_0}$, $n\in \mathbb N$, and $s\geq s_0$, then 
\begin{align*}
\frac{\sigma_s(an+b)}{\sigma_s(cn+d)} &= \frac{\sigma_s(an+b)}{(an+b)^s}\left(\frac{an+b}{cn+d}\right)^s\frac{(cn+d)^s}{\sigma_s(cn+d)}\\
&\leq \zeta(s)\left(\frac{a+b}{c+d}\right)^s \leq \zeta(s_0)\left(\frac{a+b}{c+d}\right)^{s_0} < 1.
\end{align*}
If $a+b < (c+d)(1-1/s_0)$, then the above implies that for $s\geq s_0$ and $n\in \mathbb N$,
$$
\frac{\sigma_s(an+b)}{\sigma_s(cn+d)} \leq \zeta(s_0)\left(\frac{a+b}{c+d}\right)^{s_0} \leq \frac{s_0}{s_0-1}\left(\frac{a+b}{c+d}\right) < 1.
$$
The rest is easy, so the proof is complete.
\end{proof}
 
\begin{example}\label{example3.9race}
By using a computer, we see that $\sigma(2n+5) < \sigma(6n+17)$ for all $n\leq 10^6$ while 
$$
\sigma_{1/2}(2n+5) < \sigma_{1/2}(6n+17)\quad\text{ for $n\leq 4$}
$$
and the inequality changes to $>$ when $n=5$. By Corollary \ref{sumofdfthm204}, there are infinitely many $n\in \mathbb N$ for which these inequalities are both $<$ or both $>$, but we do not know if there exist infinitely many $n$ for which these inequality are opposite. By Theorem \ref{sumofdfthm3.8}, we have $\sigma_s(2n+5) < \sigma_s(6n+17)$ and $\sigma_s(5n+4) < \sigma_s(6n+7)$ for all $n\in \mathbb N$ and $s\geq 3$. However, we do not know if there are infinitely many $m\in \mathbb N$ for which the inequalities $\sigma(2m+5) > \sigma(6m+17)$ and $\sigma(5m+4) > \sigma(6m+7)$ simultaneously hold.
\end{example}

\begin{problem}\label{problem1.1}
Let $B = \{s_1, s_2, \ldots, s_k\}$ and $C = \{s_{k+1}, s_{k+2}, \ldots, s_{k+\ell}\}$ be subsets of $[-1,1]$ and let $a, b, c, d$ satisfy the condition (A). By Corollary \ref{sumofdfthm204}, there are infinitely many $m$ and $n$ such that 
$$
\sigma_s(am+b)<\sigma_s(cm+d)\;\text{and}\;\sigma_s(an+b)>\sigma_s(cn+d)\;\text{for all $s\in B\cup C$.}
$$
Are there infinitely many $n$ for which $\sigma_s(an+b)<\sigma_s(cn+d)$ for all $s\in B$ and $\sigma_s(an+b)>\sigma_s(cn+d)$ for all $s\in C$? Perhaps, the answer depends on $B$ and $C$. If $B, C \subseteq [1-\varepsilon,1]$ where $0<\varepsilon<1$ is very small and all elements of $B$ are less than every element of $C$, is the above statement true? Similarly, suppose $c_2, d_2$ satisfy the condition (A) and $(a,b)$, $(c,d)$, $(c_2,d_2)$ are linearly independent over $\mathbb Q$. We know that there are infinitely many $m, n\in\mathbb N$ for which $\sigma_s(am+b)<\sigma_s(cm+d)$ and $\sigma_s(cn+d)<\sigma_s(c_2n+d_2)$. Since $m$ and $n$ above may be different, it is natural to ask if there are infinitely many $r\in \mathbb N$ such that $\sigma_s(ar+b)<\sigma_s(cr+d)<\sigma_s(c_2r+d_2)$. If $(a,b), (c,d), (c_2,d_2), \ldots, (c_\ell,d_\ell)$ are linearly independent over $\mathbb Q$, can we extend the above inequality to 
$$
\sigma_s(ar+b)<\sigma_s(cr+d)<\sigma_s(c_2r+d_2)<\cdots <\sigma_s(c_\ell r+d_\ell)?
$$
\end{problem}

\begin{problem}\label{problem1}
Suppose $0\leq s\leq 1$ and $a_i, b_i, c_i, d_i$ are nonnegative integers, $a_i, c_i>0$ and $a_id_i-b_ic_i\neq 0$, for each $i = 1, 2, \ldots, \ell$. Are there infinitely many $n\in \mathbb N$ such that $\sigma_{s}(a_in+b_i) < \sigma_{s}(c_in+d_i)$ for all $i\in \{1, 2, \ldots, \ell\}$? If $I, J\subseteq \{1, 2, \ldots, \ell\}$ are disjoint, are there infinitely many $n\in \mathbb N$ for which 
$$
\sigma_{s}(a_in+b_i) < \sigma_{s}(c_in+d_i)\;\text{and}\;\sigma_{s}(a_jn+b_j) > \sigma_{s}(c_jn+d_j)
$$
 for all $i\in I$ and $j\in J$? Are they true if the set of all $(a_i,b_i)$ and $(c_i,d_i)$ are linearly independent over $\mathbb Q$? Perhaps, there exist some kind of admissible sets of $a_i$, $b_i$, $c_i$, $d_i$ to guarantee that one of the above are true. Goldston, Graham, Pintz, and Yildirim \cite{GGPY}, and De Koninck and Luca \cite{DLu} solved similar problems. Their ideas may be useful in solving Problems \ref{problem1.1} and \ref{problem1} too.
\end{problem}

\begin{problem}
What is the infimum of $s_0$ such that $\sigma_s(2n+5) < \sigma_s(6n+17)$ and $\sigma_s(5n+4) < \sigma_s(6n+7)$ for all $n\in \mathbb N$ and $s\geq s_0$? By Example \ref{example3.9race} such the infimum is $\leq 3$. In general, if $a_i, b_i, c_i, d_i$ satisfy the conditions of Theorem \ref{sumofdfthm3.8} for all $i = 1, 2, \ldots, \ell$, can we determine the infimum of $s_0$ such that $\sigma_{s}(a_in+b_i) < \sigma_{s}(c_in+d_i)$ for all $s\geq s_0$, $n\in \mathbb N$, and $i = 1, 2, \ldots, \ell$?
\end{problem}

We consider (IE2) in the next subsection. We show that for each $s_0>1$, we can find integers $a$, $b$, $c$, $d$ such that for all $s\geq s_0$, $\sigma_s(an+b) < \sigma_s(cn+d)$ for an arbitrarily long string of consecutive integers $n = 1, 2, \ldots, M$ and $\sigma_s(an+b) > \sigma_s(cn+d)$ for all large $n\geq N$. In addition, if we sacrifice the uniformity of $s$, we can force $N$ to be $M+1$. 
\subsection{The second case for $|s|>1$}

\begin{theorem}\label{sumofdfthmA}
The following statements hold. 
\begin{itemize}
\item[{\rm(i)}] Let $s_0>1$ and $M\in \mathbb N$ be given. Suppose $a$, $b$, $c$, $d$ are integers satisfying $b\geq 0$, $c\geq 1$, $a > c\zeta(s_0)$, and 
$$
d\geq \zeta(s_0)b+(M+1)(a\zeta(s_0)-c).
$$
Then $a>c>0$, $ad-bc>0$, and $\sigma_s(an+b) < \sigma_s(cn+d)$ for all $s\geq s_0$ and $n = 1, 2, \ldots, M$ and there exists $N\in \mathbb N$ such that $\sigma_s(an+b) > \sigma_s(cn+d)$ for all $s\geq s_0$ and $n\geq N$.
\item[{\rm(ii)}] Let $M\in \mathbb N$ be given. Suppose $a$, $b$, $c$, $d$ are integers satisfying $c > b\geq 1$, $a>2c$, $d = (M+q)(a-c)+b$,
 where $q = q_1/q_2$, $0<q_1<q_2$, $(q_1,q_2) = 1$, and $q_2\mid a-c$. Then $a>c>0$, $ad-bc > 0$, and there exists $s_0>1$ such that $\sigma_s(an+b) < \sigma_s(cn+d)$ for all $s\geq s_0$ and $n = 1, 2, \ldots, M$ and $\sigma_s(an+b) > \sigma_s(cn+d)$ for all $s\geq s_0$ and $n\geq M+1$.  
\end{itemize}
\end{theorem}
\begin{remark}\label{sumofdfremA2}
Obviously, there are infinitely many integers $a, b, c, d$ satisfying the assumption of Theorem \ref{sumofdfthmA}. So we can find as many examples for (IE2) as we like. The integer $N$ in Theorem \ref{sumofdfthmA}{\rm(i)} may or may not be equal to $M+1$. So {\rm(ii)} does not immediately follow from {\rm(i)}. In addition, $s_0$ and $M$ in {\rm(i)} are given independently while $s_0$ in {\rm(ii)} may depend on $M$ and $a$, $b$, $c$, $d$. We do not know whether it is possible to obtain the result as in {\rm(ii)} but $M$ and $s_0$ are independent variables. 
\end{remark}

\begin{proof}[Proof of Theorem \ref{sumofdfthmA}.]
We first prove (i). Observe that 
$$
a\zeta(s_0)-c>a-c>0,\; d-b>d-\zeta(s_0)b > 0,\; ad-bc>0,
$$
 and 
$$
\frac{d-b}{a-c} \geq \frac{d-\zeta(s_0)b}{a\zeta(s_0)-c} > M.
$$
The inequality $(d-b)/(a-c) > M$ implies $an+b < cn+d$ for all $n\leq M$ and the inequality $(d-\zeta(s_0)b)/(a\zeta(s_0)-c) > M$ leads to $\zeta(s_0)(aM+b) < cM+d$. Since $ad-bc > 0$, the sequence $((an+b)/(cn+d))_{n\geq 1}$ is increasing and therefore 
$$
\frac{an+b}{cn+d} \leq \frac{aM+b}{cM+d} < \frac{1}{\zeta(s_0)} < 1\quad\text{for all $n\leq M$}.
$$
Hence for $s\geq s_0$ and $n\leq M$, we obtain 
\begin{align}\label{sumofdfthmAeq1}
\frac{\sigma_s(an+b)}{\sigma_s(cn+d)} &=\frac{\sigma_s(an+b)}{(an+b)^s}\left(\frac{an+b}{cn+d}\right)^s\frac{(cn+d)^s}{\sigma_s(cn+d)}\\
&\leq \zeta(s)\left(\frac{aM+b}{cM+d}\right)^s \leq \zeta(s_0)\left(\frac{aM+b}{cM+d}\right) < 1,\notag
\end{align}
which implies $\sigma_s(an+b) < \sigma_s(cn+d)$. It remains to show that when $n$ is large enough, the inequality reverse. From \eqref{sumofdfthmAeq1}, we see that for $s\geq s_0$
$$
\frac{\sigma_s(an+b)}{\sigma_s(cn+d)} \geq \left(\frac{an+b}{cn+d}\right)^s\frac{1}{\zeta(s)}
$$
which converges, as $n\to \infty$, to
$$
\left(\frac ac\right)^s\frac{1}{\zeta(s)} > \frac{\zeta(s_0)^s}{\zeta(s)} > 1.
$$  
So there exists $N\in \mathbb N$ such that $\sigma_s(an+b) > \sigma_s(cn+d)$ for all $n\geq N$. This $N$ can be chosen uniformly for all $s\in [s_0,\infty)$ in the sense that it depends on $s_0$, $a$, $b$, $c$, $d$ but not on $s$ as follows. As $n\to\infty$, we have $\left(\frac{an+b}{cn+d}\right)^{s_0} \to \left(\frac ac\right)^{s_0} > \zeta(s_0)^{s_0}$. Then there exists $N\in \mathbb N$ such that if $n\geq N$, then $\left(\frac{an+b}{cn+d}\right)^{s_0} > \zeta(s_0)^{s_0}$, and so for all $s\geq s_0$ and $n\geq N$, we have $\frac{an+b}{cn+d}>1$ and
$$
\frac{\sigma_s(an+b)}{\sigma_s(cn+d)}\geq\left(\frac{an+b}{cn+d}\right)^s\frac{1}{\zeta(s)}\geq\left(\frac{an+b}{cn+d}\right)^{s_0}\frac{1}{\zeta(s)}>\frac{\zeta(s_0)^s}{\zeta(s)}>1.
$$
Next, we prove (ii). The conditions on $q_1$ and $q_2$ make sure that 
$$
\text{$q\in (0,1)$, $d\in \mathbb Z^+$, and $M<(d-b)/(a-c) < M+1$.}
$$
It is also obvious that $a>c>0$, $d>b$, and $ad-bc>0$. Let $\alpha_0 = (d-b)/b$. Then $b\alpha_0 = (M+q)(a-c) > c > b$. So $\alpha_0>1$. Let $\alpha = \min\{\alpha_0,2\}$. Define functions $f, g:[1,\alpha]\to\mathbb R$ by 
$$
f(x) = \frac{d-bx}{ax-c}\quad\text{and}\quad g(x) = \frac{dx-b}{a-cx}.
$$
Then for $x\in (1,\alpha]$, we have 
\begin{align*}
dx-b &> d-bx \geq d-b\alpha \geq d-b\alpha_0 = d-(d-b) = b > 0\quad\text{and}\\
ax-c &> a-cx \geq a-c\alpha \geq a-2c > 0.
\end{align*}
Therefore $0 < f(x) < g(x)$ for all $x\in (1,\alpha]$ and $M<f(1) = g(1) < M+1$. By using the usual method in calculus, it is easy to verify that $f$ and $g$ are continuous on $[1,\alpha]$, $f$ is decreasing on $[1,\alpha]$, and $g$ is increasing on $[1,\alpha]$. Since $f(1) > M$, $\zeta(s) > 1$, and $\zeta(s)\to 1$ as $s\to \infty$, there exists a real number $s_1>1$ such that $\zeta(s)\in (1,\alpha)$ and $f(\zeta(s))>M$ for all $s\geq s_1$. Similarly, since $M+1 > g(1)$, there is a real number $s_2>1$ such that $\zeta(s)\in (1,\alpha)$ and $g(\zeta(s)) < M+1$ for all $s\geq s_2$. Let $s_0 = \max\{s_1,s_2\}$. Then $s_0 > 1$, 
\begin{equation}\label{sumofdfthmAeq2}
\zeta(s)\in (1,\alpha)\quad\text{and}\quad M < f(\zeta(s)) < g(\zeta(s)) < M+1\quad\text{for all $s\geq s_0$}.
\end{equation} 
Next, we show that $\sigma_s(an+b) < \sigma_s(cn+d)$ for all $s\geq s_0$ and $n\leq M$. Similar to the proof of the first part, the sequence $\left((an+b)/(cn+d)\right)_{n\geq 1}$ is increasing and $aM+b < cM+d$. In addition, the inequality $M<f(\zeta(s))$ in \eqref{sumofdfthmAeq2} implies that $\zeta(s)(aM+b) < cM+d$. From these, we obtain, for all $s\geq s_0$ and $n\leq M$,
$$
\frac{\sigma_s(an+b)}{\sigma_s(cn+d)} \leq \zeta(s)\left(\frac{an+b}{cn+d}\right)^s \leq \zeta(s)\left(\frac{aM+b}{cM+d}\right)^{s} < \zeta(s)\left(\frac{aM+b}{cM+d}\right) < 1, 
$$ 
as desired. Similarly, for $s\geq s_0$, we have $\zeta(s) < \alpha \leq 2$, $a-c\zeta(s)>0$, and $g(\zeta(s)) < M+1$, which implies $\zeta(s)(c(M+1)+d) < a(M+1)+b$. Therefore, for all $s\geq s_0$ and $n\geq M+1$,
$$
\frac{\sigma_s(an+b)}{\sigma_s(cn+d)} \geq \left(\frac{an+b}{cn+d}\right)^s\frac{1}{\zeta(s)} \geq \left(\frac{a(M+1)+b}{c(M+1)+d}\right)^{s}\frac{1}{\zeta(s)} > \zeta(s)^{s-1} > 1.
$$
This completes the proof.
\end{proof}
By changing the role of $a, b, c, d$, we obtain the following corollary.
\begin{corollary}\label{sumofdfcorA3}
For each $s_0>1$ and $M\in \mathbb N$, there are integers $a, b, c, d, N$ satisfying $b, d\geq 0$, $0<a<c$, $ad-bc<0$, $N\geq M+1$ such that 
\begin{align*}
&\sigma_s(an+b) > \sigma_s(cn+d) \quad\text{for all $s\geq s_0$ and $n\leq M$}\\
&\sigma_s(an+b) < \sigma_s(cn+d) \quad\text{for all $s\geq s_0$ and $n\geq N$}.
\end{align*}
Furthermore, if $M\in \mathbb N$ is given, we can find the integers $a, b, c, d$ and real number $s_0 > 1$ as the above with the additional property that $N$ can be chosen to be $M+1$.
\end{corollary}
\begin{example}\label{sumofdfexamA4}
If $s_0 = 2$ and $M = 999999$, then we obtain by Theorem \ref{sumofdfthmA}(i) that we can choose $c=2$, $b=1$, $a=5$, and $d = 6224673$ to obtain $\sigma_s(5n+1) < \sigma_s(2n+6224673)$ for all $s\geq 2$ and $n\leq 999999$ and $\sigma_s(5n+1) > \sigma_s(2n+6224673)$ for all $s\geq 2$ and $n\geq N$ for some large $N$. Suppose $M = 9999$, $c=2$, $b=1$, $a=5$, $q = 1/3$, and $d=29999$. By the proof of Theorem \ref{sumofdfthmA}(ii), we need to select $s_1, s_2 > 1$ so that 
$$
\frac{d-b\zeta(s_1)}{a\zeta(s_1)-c} > M\quad\text{and}\quad \frac{d\zeta(s_2)-b}{a-c\zeta(s_2)} < M+1,
$$
that is, $\zeta(s_1)<\frac{49997}{49996}$ and $\zeta(s_2) < \frac{50001}{49999}$. We can use Wolfram Alpha to find an estimate for $\zeta(s)$ by simply writing zeta$(s)$ and clicking enter. The decimal expansion of $\zeta(16)$ is also given as Sequence A013674 in the Online Encyclopedia of Integer Sequences. We see that we can take $s_1 = s_2 = 16$. So we choose $s_0 = 16$ and obtain that 
\begin{align*}
\sigma_s(5n+1) < \sigma_s(2n+29999)\quad\text{for all $s\geq 16$ and $n\leq 9999$},\\
\sigma_s(5n+1) > \sigma_s(2n+29999)\quad\text{for all $s\geq 16$ and $n\geq 10000$}.
\end{align*}
The number $s_0 = 16$ is not optimal but the results change when $s$ is smaller than $13$. For each $s>1$, let $h(s)$ be the smallest positive integer $n$ such that $\sigma_s(5n+1) > \sigma_s(2n+29999)$. So $h(s) = 10000$ for all $s\geq 16$. The values of $h(2), h(3), \ldots, h(15)$ are as follows: 
\begin{align*}
h(s) &= 10,000\quad\text{for $s\in \{13, 14, 15, 16\}$},\quad h(s) = 9999\quad\text{for $s\in\{10, 11, 12\}$},\\
h(9) &= 9997, \quad h(8) = 9991, \quad h(7) = 9981, \quad h(6) = 9995 \\
h(5) &= 9883, \quad h(4) = 9691, \quad h(3) = 9115, \quad h(2) = 7207.
\end{align*}
\end{example}

\begin{problem}
From Remark \ref{sumofdfremA2}, prove or disprove that if $M$ is given and is very large, then there exists $s_0 > 1$ such that if $1< s\leq s_0$, then there are no integers $a, b, c, d$ satisfying $a>c>0$, $b, d\geq 0$, $ad-bc\neq 0$, $\sigma_s(an+b) < \sigma_s(cn+d)$ for all $n\leq M$ and $\sigma_s(an+b) > \sigma_s(cn+d)$ for all $n\geq M+1$.
\end{problem}

\begin{problem}
From Example \ref{sumofdfexamA4}, it may be interesting to study the behavior of the function $h$ on $(1,16]$. In general, suppose $a, b, c, d, M, s_0$ satisfy the conditions in Theorem \ref{sumofdfthmA}(ii), and $h(s)$ is the smallest positive integer $n$ such that $\sigma_s(an+b)-\sigma_s(cn+d)$ changes sign. Then $h(s) = M+1$ for all $s\geq s_0$. Can we describe the behavior of $h$ on $(1,s_0]$?
\end{problem}

Recall that $\left\lceil s\right\rceil$ is the smallest integer larger than or equal to $s$. Next, we construct an example for (IE3).

\subsection{The third case for $|s|>1$}

\begin{theorem}\label{newthmA}
Let $s > 0$ and $n = \left\lceil s\right\rceil$ be given. If $p$ is a prime larger than $1+n2^{n+1}$, then $\sigma_s(p-1)>\sigma_s(p)$ and $\sigma_s(p) < \sigma_s(p+1)$. In particular, there are infinitely many $m, r\in \mathbb N$ such that 
$$
\sigma_s(m) > \sigma_s(m+1)\; \text{and}\; \sigma_s(r) < \sigma_r(r+1).
$$
\end{theorem}
\begin{proof}
Suppose $p>1+n2^{n+1}$. Then $\sigma_s(p+1)\geq 1+(p+1)^s>1+p^s = \sigma_s(p)$. For the other inequality, we observe that $n/(p-1) < 1/2^{n+1}$ and 
\begin{align*}
\left(1+\frac{1}{p-1}\right)^s &\leq \left(1+\frac{1}{p-1}\right)^n = \sum_{k=0}^n{n\choose k}\left(\frac{1}{p-1}\right)^k\\
&\leq\sum_{k=0}^n\left(\frac{n}{p-1}\right)^k < \sum_{k=0}^n\left(\frac{1}{2^{n+1}}\right)^k <1+\sum_{k=1}^\infty\frac{1}{2^{n+k}}\\
&= 1+\frac{1}{2^n} \leq 1+\frac{1}{2^s}.
\end{align*}
Then $\sigma_s(p-1)-\sigma_s(p)\geq 1+((p-1)/2)^s+(p-1)^s-1-p^s$, which is equal to
$$
(p-1)^s\left(1+\frac{1}{2^s}-\left(1+\frac{1}{p-1}\right)^s\right) > 0.
$$
Therefore $\sigma_s(p-1) > \sigma_s(p)$ and $\sigma_s(p) < \sigma_s(p+1)$, as required. Since there are infinitely many such $p$, we can take $m=p-1$ and $r=p$ to obtain infinitely many $m, r\in \mathbb N$ as desired.
\end{proof}

By Theorem \ref{newthmB}, in order to construct an example for (IE3), it is easier (may be necessary) to consider the case $a=c$. Then we have a generalization of Theorem \ref{newthmA} as follows.
\begin{theorem}\label{newthmC}
Suppose $s>0$, $a\in \mathbb N$, $b, d$ are distinct nonnegative integers, and $(a,b)=(a,d)=1$. Then there are infinitely many $m, n\in \mathbb N$ such that $\sigma_s(am+b) < \sigma_s(am+d)$ and $\sigma_s(an+b) > \sigma_s(an+d)$.
\end{theorem}
\begin{proof}
The statement is symmetric with respect to $b$ and $d$, so we can assume without loss of generality that $b<d$. Since $(a,b)=1$, there are infinitely many $m\in \mathbb N$ such that $am+b$ is prime, and therefore 
$$
\sigma_s(am+b) = 1+(am+b)^s < 1+(am+d)^s \leq \sigma_s(am+d).
$$
Next, let $\ell = d-b$, $k = \left\lceil s\right\rceil$, and $q$ a prime with $q\nmid a\ell$. Then $(\ell,q)=(a,q)=(a,d) = 1$. By Chinese remainder theorem and Dirichlet's theorem for primes in arithmetic progressions, there are infinitely many primes $p$ such that 
$$
\text{$p\equiv \ell\pmod q$, $p\equiv d\pmod a$, $p>d$, and $p>\ell+\ell kq^{k+1}$.}
$$
For each such $p$, let $n=n_p=(p-d)/a$. Then $\sigma_s(an+d) = \sigma_s(p)=1+p^s$, $an+b = p-d+b = p-\ell$, and so $\sigma_s(an+b)\geq 1+(p-\ell)^s+((p-\ell)/q)^s$. Then $\sigma_s(an+b)-\sigma_s(an+d)\geq (p-\ell)^s(1+1/q^s)-p^s$. To show that $\sigma_s(an+b)>\sigma_s(an+d)$, it is enough to show that $1+1/q^s > p^s/(p-\ell)^s$. We have
\begin{align*}
\left(\frac{p}{p-\ell}\right)^s &= \left(1+\frac{\ell}{p-\ell}\right)^s \leq \left(1+\frac{\ell}{p-\ell}\right)^k\\
&=\sum_{j=0}^k{k\choose j}\left(\frac{\ell}{p-\ell}\right)^j \leq \sum_{j=0}^k\left(\frac{k\ell}{p-\ell}\right)^j.
\end{align*}
Since $k\ell/(p-\ell) < 1/q^{k+1}$, the above is less than 
$$
1+\frac{1}{q^{k+1}}+\frac{1}{q^{k+2}}+\frac{1}{q^{k+3}}+\cdots = 1+\frac{1}{(q-1)q^k} \leq 1+\frac{1}{q^{k}} \leq 1+\frac{1}{q^s}.
$$
So the proof is complete.
\end{proof}

\begin{problem}
Suppose $s>1$ and $a, b, c, d$ are positive integers. If 
$$
\text{$\sigma_s(an+b)-\sigma_s(cn+d)$ changes sign infinitely often,}
$$
is it true that $a=c$? If it is not true, assuming further that $s$ is large and $(a,b) = (c,d) = 1$, can we conclude that $a=c$? 
\end{problem}

\begin{problem}
By Theorem \ref{sumofdfthm201}, $\sigma(30n+1) - \sigma(30n)$ has infinitely many sign changes but it may be interesting to know more about this. For example, is it true that 
$$
\text{(i)}\quad\sum_{\substack{n\leq m\\ \sigma(30n+1)<\sigma(30n)}}1 > \sum_{\substack{n\leq m\\ \sigma(30n+1)>\sigma(30n)}}1 \quad\text{for all $m\in \mathbb N$?}
$$ 
Is it true that 
$$
\text{(ii)}\quad\sum_{n\leq m} \sigma(30n) > \sum_{n\leq m} \sigma(30n+1) \quad\text{for all $m\in \mathbb N$?}
$$ 

Numerical evidence suggests that they are true but we currently do not have a proof. If they are not true, do the sign reverse for infinitely many $m$? We can replace $30n+1$ and $30n$ by $an+b$ and $cn+d$ and we can also change $\sigma$ to $\sigma_s$, or any other function of interest. Perhaps, it is true that 
$$
\text{(iii)}\quad \sum_{\substack{n\leq m\\\sigma(30n+1)<\sigma(30n)}} \frac1n> \sum_{\substack{n\leq m\\\sigma(30n+1)>\sigma(30n)}} \frac1n. 
$$
Maybe, the left-hand side of (iii) divided by $\log m$ is closed to $1$ as $m\to\infty$ while the right-hand side of (iii) divided by $\log m$ is closed to $0$ as $m\to\infty$. This question is motivated by those on primes. So it should be useful to read, for example, the articles by Bays and Hudson \cite{BHu}, Knapowski and Turan \cite{KTu, KTu1, KTu2}, and Meng \cite{Meng2}. See also \cite{PPo2, Pon2021long, PSu} and some of our future articles for more information on palindromes and the comparison between the number of palindromes in different bases.
\end{problem}

\begin{problem}
Let $J_s$ and $\varphi_s$ be the arithmetic functions that are defined for all $n\in \mathbb N$ by 
$$
J_s(n) = n^s\prod_{p\mid n}\left(1-\frac{1}{p^s}\right)\quad\text{and}\quad \varphi_s(n) = \sum_{\substack{1\leq m\leq n\\(m,n)=1}}m^s.
$$
So $J_1(n) = \varphi(n) = \varphi_0(n)$ and therefore $J_s$ and $\varphi_s$ are generalizations of the Euler function $\varphi$. What are the corresponding results for $J_s(an+b)$, $J_s(cn+d)$, $\varphi_s(an+b)$, and $\varphi_s(cn+d)$. There are other generalizations of $\varphi$ which may be of interest. The inequalities related to $\sigma_s$, $\varphi_s$, and $J_s$ over the Fibonacci numbers are also investigates in \cite{JPo, Luca}. For some recent articles related to divisibility properties of the Fibonacci numbers, see for example in \cite{OPo2020, OPo2021, PPo2020}.
\end{problem}

\begin{problem}\label{problem12sign}
We already showed that the following three cases may occur:
\begin{itemize}
\item[(i)] $\sigma_s(an+b) - \sigma_s(cn+d)$ never changes sign,
\item[(ii)] $\sigma_s(an+b) - \sigma_s(cn+d)$ changes sign exactly once,
\item[(iii)] $\sigma_s(an+b) - \sigma_s(cn+d)$ changes sign infinitely many times.
\end{itemize}
Is it possible to construct an example for which $\sigma_s(an+b) - \sigma_s(cn+d)$ changes sign exactly two times? If $k$ is a given positive integer, can we find $s\in \mathbb R$, $a, b, c, d\geq 0$ such that there are exactly $k$ sign changes in $\sigma_s(an+b) - \sigma_s(cn+d)$?
\end{problem}

We can replace $\sigma$ by $\omega$ and $\Omega$ and obtain the same results in Theorem \ref{sumofdfthm201}. Here $\omega(n)$ is the number of distinct prime divisors  of $n$ and $\Omega(n)$ is the number of prime powers dividing $n$. We record it as a theorem.
\begin{theorem}
Let $a, b, c, d$ satisfy the condition {\rm(A)}. Then there exists a strictly increasing sequence $(n_k)$ of positive integers such that  
$$
{\rm(i)} \lim_{k\to\infty}\frac{\omega(an_k+b)}{\omega(cn_k+d)} = \lim_{k\to\infty}\frac{\Omega(an_k+b)}{\Omega(cn_k+d)} = 0.
$$
In particular,   
\begin{align*}
&\text{{\rm(ii)} $\displaystyle \liminf_{n\to\infty}\frac{\omega(an+b)}{\omega(cn+d)} = \liminf_{n\to\infty}\frac{\Omega(an+b)}{\Omega(cn+d)} = 0$};\\
&\text{{\rm(iii)} $\displaystyle \limsup_{n\to\infty}\frac{\omega(an+b)}{\omega(cn+d)} = \limsup_{n\to\infty}\frac{\Omega(an+b)}{\Omega(cn+d)} = \infty$.}
\end{align*}
\end{theorem}
\begin{proof}
We construct the sequence $(n_k)$ in exactly the same way as in the proof of Theorem \ref{sumofdfthm201}. Then 
\begin{align*}
\omega(an_k+b) &= \omega(\delta q_k)\leq 1+\omega(|ad-bc|),\\
\Omega(an_k+b) &= \Omega(\delta q_k)\leq 1+\Omega(|ad-bc|),\\
\omega(cn_k+d) &= \omega(m_ky_k)\geq \omega(m_k)\to\infty\;\text{as $k\to\infty$},\\
\Omega(cn_k+d) &= \Omega(m_ky_k)\geq \Omega(m_k)\to\infty\;\text{as $k\to\infty$}.
\end{align*}
These imply the desired results.
\end{proof}

\begin{problem}
We can extend the functions $\omega$ and $\Omega$ by defining 
$$
\omega_s(n) = \sum_{p\mid n}p^s\quad\text{and}\quad \Omega_s(n) = \sum_{p^k\parallel n} kp^s\quad\text{for all $n\in \mathbb N$}.
$$
What are the corresponding results for $\omega_s$ and $\Omega_s$ over arithmetic progressions? Some results on $\omega(F_n)$ and $\omega(L_n)$ are obtained in \cite{BLMS} and \cite{ Pon16}.
\end{problem}

\begin{problem}
It may be interesting to compare $f(A_{an+b})$ and $f(A_{cn+d})$, where $f$ is an arithmetic function and $(A_n)_{n\geq 1}$ is an integer sequence which is of interest. For example, suppose $F_n$ is the $n$th Fibonacci number, can we say anything about the behavior of $\sigma_s(F_{an+b})$ and $\sigma_s(F_{cn+d})$? The order (or rank) of appearance of $n$ in the Fibonacci sequence, denoted by $z(n)$, is the smallest positive integer $m$ such that $n\mid F_m$. The function $z$ is not multiplicative but is closed to being multiplicative because 
\begin{align*}
z(\lcm[m,n]) &= \lcm[z(m),z(n)],\;\text{ and so}\\
z(p_1^{a_1}p_2^{a_2}\cdots p_k^{a_k}) &= \lcm[z(p_1^{a_1}), z(p_2^{a_2}), \ldots, z(p_k^{a_k})],
\end{align*}
 where $p_1, p_2, \ldots, p_k$ are distinct primes and $a_1, a_2, \ldots, a_k$ are positive integers. Can we compare $z(an+b)$ and $z(cn+d)$? Some formulas concerning $z(n)$ when $n$ is of special forms are shown in \cite{KPo2, Pon8, Pon20} and in the references of these articles. We may also replace $z$ by any other interesting arithmetic functions.  
\end{problem}

\begin{problem}
Define $\sigma_s(n,q,a)$ by 
$$
\sigma_s(n,q,a) = \sum_{\substack{d\mid n\\d\equiv a\Mod q}}d^s.
$$
What are the results on the inequalities between $\sigma_s(n,q,a_1)$ and $\sigma_s(n,q,a_2)$? For example, if $q=2$, we consider the inequalities between the sum of odd divisors of $n$ and even divisors of $n$. Are there an infinite number of sign changes in $\sigma_s(n,2,0)-\sigma_s(n,2,1)$ or in $\sum_{n\leq x}\sigma_s(n,2,0) - \sum_{n\leq x}\sigma_s(n,2,1)$? The answers may generally depend on $s, n, q, a$. We can define $\omega_s(n,q,a)$ and $\Omega_s(n,q,a)$ in a similar way. For some results in this direction, see for example in the work of Khan \cite{Khan}, Liu, Shparlinski, and Zhang \cite{LSZ}, Meng \cite{Meng2} and Pongsriiam and Vaughan \cite{Pon10, Pon3, PVa},  and references therein. If we are interested in other kinds of divisors, we may consider the comparison between the sums of those divisors. Define $\sigma_{s,A}(n) = \sum_{d\in A(n)}d^s$ for all $n\in \mathbb N$, where $A(n)$ is, for example, the squarefree divisors of $n$ or the unitary divisors of $n$. There are many other possible research problems. We leave them to the reader's imagination.
\end{problem}

\subsection*{Acknowledgements}
This project is funded by the National Research Council of Thailand (NRCT), Grant Number NRCT5-RSA63021-02.

%\normalsize

\end{document}